\documentclass{amsart} 

\usepackage{amssymb}
\usepackage{tikz}
\usepackage[shortlabels]{enumitem}
\usepackage{cleveref}

\newcommand{\seq}[1]{{\left\langle{#1}\right\rangle}}
 
\renewcommand \leq {\leqslant}

\renewcommand \le {\leqslant}
\renewcommand \ge {\geqslant}
\newcommand{\converge}{\!\!\downarrow}

\renewcommand \phi {\varphi}

\newcommand{\cat}{\widehat{\phantom{\alpha}}}

\newcommand{\abar}{\bar{a}}
\newcommand{\bbar}{\bar{b}}
\newcommand{\cbar}{\bar{c}}
\newcommand{\dbar}{\bar{d}}

\newcommand{\A}{\mathcal{A}}
\newcommand{\B}{\mathcal{B}}
\newcommand{\C}{\mathcal{C}}
\newcommand{\M}{\mathcal{M}}

\newcommand{\Z}{\mathbb{Z}}

\newcommand{\zero}{{\mathbf 0}}
\newcommand{\zeroj}{{\mathbf 0}'}
\newcommand{\zerojj}{{\mathbf 0}''}
\newcommand{\zeron}{{\mathbf 0^{(n)}}}
\newcommand{\degd}{{\mathbf d}}

\newcommand{\wock}{\omega_1^{\textup{ck}}}

\newcommand{\outcome}{\texttt{outcome}}

\DeclareMathOperator \dom{dom}

\theoremstyle{plain}
\newtheorem{theorem}{Theorem}
\newtheorem{proposition}[theorem]{Proposition}

\newtheorem{corollary}[theorem]{Corollary}

\newcounter{claimCounter}[theorem]

\newtheorem{claim}[claimCounter]{Claim}

\newcounter{subClaimCounter}[claimCounter]

\newtheorem{subclaim}[subClaimCounter]{Claim}

\theoremstyle{definition}

\begin{document}

\title[Coding in the aut.\ group of a comp.\ cat.\ structure]{Coding in the automorphism group of a computably categorical structure}
\author{Dan Turetsky}
\address{School of Mathematics and Statistics\\ Victoria University of Wellington\\ Wellington, New Zealand}
\email{dan.turetsky@vuw.ac.nz}

\begin{abstract}
Using new techniques for controlling the categoricity spectrum of a structure, we construct a structure with degree of categoricity but infinite spectral dimension, answering a question of Bazhenov, Kalimullin and Yamaleev.  Using the same techniques, we construct a computably categorical structure of non-computable Scott rank.
\end{abstract}

\maketitle

\section{Introduction}

Isomorphism problems for computable structures are one of the main topics in computable model theory.  This is a broadly defined topic concerned with the complexity of isomorphisms between isomorphic computable structures.  Recent work has focused on degrees of categoricity and of isomorphism~\cite{FoKaMi10,CsFrSh13,AnCs16,CsHa17,Fr17,BaKaYa18}, and lowness for categoricity and for isomorphism~\cite{AnCs16,FrSo14,Fr17,FrTu18,McSt18}.  We refer the reader to the text by Ash and Knight~\cite{AshKnight} for background on computable model theory and computable ordinals.

Given isomorphic computable structures $\A$ and $\B$, the {\em degree of isomorphism} between $\A$ and $\B$ is the least Turing degree $\degd$ which can compute an isomorphism between $\A$ and $\B$, if such a degree exists.  In this case, every isomorphism between $\A$ and $\B$ computes $\degd$, so we can think of $\degd$ as being coded into the isomorphism problem of $\A$ and $\B$.

A related notion is the degree of categoricity.  Given a computable structure $\A$, the {\em degree of categoricity} of $\A$ is the least degree which computes an isomorphism between any two computable presentations of $\A$, if such a degree exists.  Fokina, Kalimullin and R.\ Miller introduced this notion and showed that every degree which is d.c.e.\ in and above some $\zeron$ is the degree of categoricity of some computable structure~\cite{FoKaMi10}.  Csima, Franklin and Shore extended this to every degree which is d.c.e.\ in and above some $\zero^{(\alpha+1)}$~\cite{CsFrSh13}.  Recently, Csima and Ng have announced that every $\Delta^0_2$ degree is the degree of categoricity for some computable structure.

When building a structure to have degree of categoricity $\degd$, one must achieve two goals: $\degd$ must be able to compute an isomorphism between any two computable presentations; and $\degd$ must be the least degree which can do this.  The natural way to achieve the second goal is to build two presentations, $\A$ and $\B$, such that every isomorphism between $\A$ and $\B$ computes $\degd$.  In other words, $\degd$ is the degree of isomorphism for the pair $\A, \B$.  When there exists such a pair of presentations, we say that $\degd$ is the {\em strong degree of categoricity} of $\A$.

The original constructions of degrees of categoricity actually constructed strong degrees of categoricity.  It is natural to wonder whether there is a degree which is the degree of categoricity of some computable structure, but is not the strong degree of categoricity of any computable structure.  This question remains open---so far, every degree which is known to be the degree of categoricity of some computable structure is also known to be the strong degree of categoricity of some computable structure.

A simpler question is whether there exists a structure $\A$ with some degree of categoricity $\degd$, but $\degd$ is not a strong degree of categoricity for $\A$.  Bazhenov, Kalimullin and Yamaleev~\cite{BaKaYa18}, and separately Csima and Stephenson~~\cite{CsSt19}, showed that this can occur.  They constructed a structure $\A$ with degree of categoricity $\zeroj$, but such that $\zeroj$ is not the degree of isomorphism for any two computable presentations of $\A$.  In both the Bazhenov-Kalimullin-Yamaleev example and the Csima-Stephenson example the structure is rigid, and there are three computable presentations, $\A_0, \A_1, \A_2$, such that if $f_1$ is the unique isomorphism between $\A_0$ and $\A_1$, and $f_2$ is the unique isomorphism between $\A_0$ and $\A_2$, then $f_1 \oplus f_2 \equiv_T \emptyset'$.  So $\zeroj$ is the degree of isomorphism for the triple $\A_0, \A_1, \A_2$, in the sense that it is the least degree able to compute an isomorphism between any two elements of the triple.

Bazhenov et al.\ observed that their methods can be generalized to produce, for any $n \in \omega$, a computable structure $\A$ such that $\zeroj$ is the degree of categoricity of $\A$, but $\zeroj$ is not the degree of isomorphism for any set of $n$ presentations $\A_0, \dots, \A_{n-1}$.  For the examples constructed using these methods, there will always be a finite set of presentations $\A_0, \dots, \A_{k-1}$ such that $\zeroj$ is the degree of isomorphism of this set.  They defined the term spectral dimension to describe this phenomenon: for a computable structure $\A$ with degree of categoricity $\degd$, the {\em spectral dimension} of $\A$ is the least cardinality of a set of presentations of $\A$ for which the degree of isomorphism is $\degd$.\footnote{The definition given here is a special case of Bazhenov et al.'s definition.  They defined spectral dimension for all computable structures, rather than just those with a degree of categoricity, using the notion of categoricity spectrum.}  They then asked whether a finite spectral dimension is necessary.  Is there a computable structure with a degree of categoricity $\degd$, but with infinite spectral dimension?  We show in \Cref{thm:weak_deg_cat} that there is such a structure, where $\degd = \zerojj$.

Changing gears slightly, a computable structure is {\em computably categorical} if it has degree of categoricity $\zero$---between any two computable copies, there is a computable isomorphism.  A computable structure $\A$ is {\em relatively $\Delta^0_\alpha$-categorical} if for any (not necessarily computable) copy $\B$ of $\A$, there is a $\Delta^0_\alpha(\B)$ isomorphism between $\B$ and $\A$.  Observe that a relatively $\Delta^0_1$-categorical structure is computably categorical.   Goncharov~\cite{Go77} showed that the converse fails.  Downey, Kach, Lempp, Lewis, Montalb\'an and Turetsky~\cite{DoKaLeLeMoTu15} showed that for every $\alpha < \wock$, there is a computably categorical structure which is not relatively $\Delta^0_\alpha$-categorical.  They used their techniques to show that the set of (indices for) computably categorical structures is $\Pi^1_1$-complete.

The structures built by Downey et al.\ are always relatively $\Delta^0_\beta$-categorical for some $\beta < \wock$.  It is thus natural to ask whether there is a computably categorical structure which is not relatively $\Delta^0_\alpha$-categorical for any $\alpha < \wock$.  This turns out to relate to several other notions.

By the Spector-Gandy theorem, every $\Pi^1_1$ formula $\forall f\, \theta(f, n)$ can be reinterpreted as $\exists f \in \Delta^1_1\, \theta'(f, n)$, where $\theta'$ is an arithmetic relation.  As $\Delta^1_1 = \bigcup_{\alpha < \wock} \Delta^0_\alpha$, this implies that every $\Pi^1_1$ set can be stratified into $\wock$ levels: if $X$ is a $\Pi^1_1$ set with $n \in X \iff \exists f \in \Delta^1_1\, \theta(n, f)$, then
\[
X = \bigcup_{\alpha < \wock} \{ n : \exists f \in \Delta^0_\alpha\, \theta(f, n)\}.
\]
For a given $\Pi^1_1$ set $X$, it can be enlightening to examine stratifications $X = \bigcup_{\alpha < \wock} X_\alpha$, where each $X_\alpha$ is uniformly $\Delta^0_\alpha$.  While the argument above always gives such a stratification, there is no known natural stratification for the set of computably categorical structures.\footnote{We acknowledge that our criteria for naturalness is vague.}  If every computably categorical structure were relatively $\Delta^0_\alpha$-categorical for some $\alpha <\wock$, this would give us such a stratification: categorize computably categorical structures based on the $\alpha$ such that they are relatively $\Delta^0_\alpha$-categorical.

Let us briefly recall the notion of the Scott rank of a structure.  For tuples $\abar, \bbar \in \A$ of the same length, define $\abar \equiv_0 \bbar$ if $\abar$ and $\bbar$ satisfy the same atomic formulae in $\A$.  For an ordinal $\alpha > 0$, define $\abar \equiv_\alpha \bbar$ if for every $\beta < \alpha$ and every $\cbar \in \A$ there exist $\dbar_0, \dbar_1 \in A$ such that $\abar\cbar \equiv_\beta \bbar\dbar_0$ and $\abar\dbar_1 \equiv_\beta \bbar\cbar$.  Define $r(\abar)$ to be the least $\alpha$ such that for all $\bbar$, $\abar \equiv_\alpha \bbar$ implies $\abar$ and $\bbar$ are in the same orbit.  By a $\Sigma^1_1$-bounding argument, for a computable structure $\A$, $r(\abar)$ will always exist and be an ordinal at most $\wock$.  The {\em Scott rank} of $\A$ is defined to be $\sup_{\abar \in \A} ( r(\abar)+1)$.

The following is a synthesis of published results and folklore.
\begin{proposition}
For a computable structure $\A$, the following are equivalent:
\begin{enumerate}
\item $\A$ is relatively $\Delta^1_1$-categorical.
\item For some $\alpha < \wock$, $\A$ is relatively $\Delta^0_\alpha$-categorical.
\item The Scott rank of $\A$ is a computable ordinal.
\item The orbits of $\A$ are uniformly hyperarithmetic.
\end{enumerate}

\end{proposition}

\begin{proof}
$(2)\Rightarrow (1)$ is immediate.  For the converse, fix a $f$ a $\Sigma^1_1$-generic permutation of $\omega$, and define $\B$ as the pullback of $\A$ along $f$, so that $f: \B \to \A$ is an isomorphism.  As $f$ is $\Sigma^1_1$-generic, $\wock(f) = \wock$~\cite{GrMo17}.  As $\B \le_T f$, $\wock(\B) = \wock$.

Since $\A$ is relatively $\Delta^1_1$-categorical, there is a $g \in \Delta^1_1(\B)$ with $g : \B \to \A$.  So $g \in \Delta^0_\alpha(\B)$ for some $\alpha < \wock(\B) = \wock$.  By Ash's proof of the characterization of relatively $\Delta^0_\alpha$-categoricity using Scott families~\cite{Ash87}, $\A$ is relatively $\Delta^0_\alpha$-categorical.  The key observation is that Ash's proof does not use the full power of $\Delta^0_\alpha$-categoricity, but just the existence of a $\Delta^0_\alpha(\B)$ isomorphism, where $\B$ is defined to be a sufficiently generic permutation of $\A$.

$(2) \Leftrightarrow (3)$ is by the standard Scott rank analysis and the characterization of relatively $\Delta^0_\alpha$-categoricity using Scott families (see Ash \& Knight section 6.7 and theorem 10.14~\cite{AshKnight}).

$(3) \Rightarrow (4)$ is by the standard Scott rank analysis, while $(4) \Rightarrow (3)$ is by $\Sigma^1_1$-bounding.
\end{proof}

In \Cref{thm:cc_high_Scott}, we construct a computably categorical structure which has non-computable Scott rank, and thus a structure which is not relatively $\Delta^1_1$-categorical, defeating the earlier potential stratification of the computably categorical structures.

\medskip

With a simple modification, the computably categorical structure we create in \Cref{thm:cc_high_Scott} can be turned into a structure of computable dimension 2.  A computable structure has {\em computable dimension 2} if it has exactly two computable copies up to computable isomorphism.  Prior constructions of such structures used a method called the special component technique~\cite{Go77,ChGoKh99,HiKhSh03}.  When a structure of computable dimension 2 is built using this technique, the two computable copies which are not computably isomorphic will be $\Delta^0_3$ isomorphic.  In contrast, for our structure of computable dimension 2, there is no hyperarithmetic isomorphism between the two copies.

The proofs of \Cref{thm:cc_high_Scott,thm:weak_deg_cat} use similar techniques, although \Cref{thm:weak_deg_cat} is the more complicated construction.  We recommend the reader begin with \Cref{thm:cc_high_Scott} as a warmup.  Both constructions are based on an extension of the techniques of Franklin and Turetsky for coding $\Pi^0_1$-classes into the automorphism group of a structure~\cite{FrTuSubmitted}, while also incorporating the methods of Downey et al.\ for constructing computably categorical structures~\cite{DoKaLeLeMoTu15}.

\section{A computably categorical structure with high Scott rank}

This section is devoted to the proof of the following theorem.

\begin{theorem}\label{thm:cc_high_Scott}
There is a computably categorical structure with non-computable Scott rank.
\end{theorem}

Fix $T \subset \omega^{<\omega}$ the tree of descending sequences through the Harrison ordering.  The only important properties of $T$ are that it is computable and that the set of extendible nodes in $T$ is not hyperarithmetic.

Our structure will have $U$ a unary relation, $(W_\sigma)_{\sigma \in \omega^{<\omega}}$ all unary relations, $(E_i)_{i \in \omega}$ all binary relations, $P$ a binary relation, $f$ a unary function, and $(V_n)_{n \in \omega}$ all unary relations.  Our structure's universe will be $\bigl([\omega]^{<\omega} \times \omega^{<\omega}\bigr) \sqcup C$, where $C$ is an infinite computable set.

The purpose of $U$ is to distinguish $C$ from the rest of the structure, so $U(x)$ will hold if and only if $x \in C$.  $P$, the $E_i$ and the $W_\sigma$ will never hold with elements of $C$, while the $V_n$ will never hold with elements outside of $C$.  The purpose of $C$, $f$ and $V_n$ is to allow us to apply labels to elements of $[\omega]^{<\omega} \times \omega^{<\omega}$ in a c.e.\ fashion.  For each $x \in C$, $f(x) \not \in C$, and for each $x \not \in C$, $f(x) = x$.  Also, for each $x \in C$, there will be a unique $n$ such that $V_n(x)$ holds.  For $y \not \in C$, we will write $S_n(y)$ for ``$\exists x \in C\, f(x) = y \wedge V_n(x)$''.  During our construction, there will be times when we wish to declare $S_n(y)$ for some $n$ and $y$.  When we do, we choose the first element $x \in C$ which has not yet been used and define $V_n(x)$ and $f(x) = y$.  We simultaneously define $\neg V_m(x)$ for all $m \neq n$.  As there will be infinitely many times we wish to make such a declaration, this will completely define the structure on $C$.  For simplicity of presentation, we may sometimes declare $S_n(y)$ and then make the same declaration again at a later stage.  In this case, we do not mean to choose a new $x$ for this declaration; instead, the duplicate declaration should simply be ignored.

For the remainder of the construction, we will not mention $U$, $C$, $f$ or the $V_n$.  Instead, we will proceed as if our structure has universe $[\omega]^{<\omega} \times \omega^{<\omega}$ and our language is $(W_\sigma)_{\sigma \in \omega^{<\omega}}$, $(E_i)_{i \in \omega}$, $P$, $(S_n)_{n \in \omega}$, and each $S_n$ only needs to be declared in a c.e.\ fashion rather than a computable fashion.  Note that this does not affect computable categoricity: if we can find a computable isomorphism $g$ between two presentations $\A_0$ and $\A_1$ of this new structure, we can extend to a computable isomorphism between the corresponding versions of the old structure -- for each $y \in \A_0$, when we see an $x$ with $f(x) = y$ and $V_n(x)$, we search for an $x'$ with $f(x') = g(y)$ and $V_n(x')$, and then define $g(x) = x'$.

With the exception of the $S_n$, we can give a complete description of our structure immediately.  For $(F, \tau), (G, \rho) \in [\omega]^{<\omega} \times \omega^{<\omega}$:
\begin{itemize}
\item $W_\sigma( (F, \tau) )$ holds iff $\sigma = \tau$.
\item $E_i( (F,\tau), (G,\rho) )$ holds iff $\tau = \rho$ and $F\triangle G = \{i\}$.
\item $P( (F,\tau), (G,\rho) )$ holds iff $\rho = \tau\cat i$ for some $i$, and one of the following holds:
\begin{itemize}
\item $i \not \in F$ and $|G|$ is even; or
\item $i \in F$ and $|G|$ is odd.
\end{itemize}
\end{itemize}
We think of $[\omega]^{<\omega}$ as an affine space acted on by $\bigoplus_{i < \omega} \Z/2$, where $F + e_i = F\triangle{i}$.  Then $E_i(F, G) \iff F + e_i = G$.  Alternatively, we can think of $[\omega]^{<\omega}$ as the vertices of an infinite dimensional cube, where there is an edge between $F$ and $G$ iff $|F\triangle G| = 1$.  If $F\triangle G = \{i\}$, we color the edge between $F$ and $G$ with color $i$, and this is represented by the relation $E_i$.  In either case, we assign a copy of this structure to every string in $\omega^{<\omega}$, and the various $W_\sigma$ let us easily identify which copy a given element belongs to.

\begin{claim}
The automorphisms of $([\omega]^{<\omega}, (E_i)_{i \in \omega})$ are precisely the maps of the form $g(F) = F\triangle H$ for some fixed $H \in [\omega]^{<\omega}$.
\end{claim}

\begin{proof}
To see that such a map is an automorphism, observe that
\begin{align*}
E_i(F,G) &\iff F\triangle G = \{i\}\\
&\iff F\triangle G \triangle \emptyset = \{i\}\\
&\iff F\triangle G\triangle (H\triangle H) = \{ i\}\\
&\iff (F\triangle H)\triangle (G\triangle H) = \{ i\}\\
&\iff E_i(g(F), g(G)).
\end{align*}

Conversely, suppose $g$ is an automorphism.  Let $H = g(\emptyset)$.  We prove by induction on $|F|$ that $g(F) = F \triangle H$.  The case $|F| = 0$ is immediate.  For $|F| > 0$, fix $i \in F$, and let $G = F - \{i\}$.  Then $E_i(F, G)$, so $E_i(g(F), g(G))$, and thus
\begin{align*}
g(F)\triangle g(G) &= \{i\}\\
g(F) &= g(G) \triangle \{i\}\\
g(F) &= (G\triangle H)\triangle \{i\}\\
g(F) &= (G\triangle \{i\})\triangle H\\
g(F) &= F\triangle H.\qedhere
\end{align*}
\end{proof}

Interpreting $[\omega]^{<\omega}$ as an affine space, it partitions into two hyperplanes perpendicular to $e_i$: $\{ F : i \not \in F\}$ is one, and $\{ F :  i \in F\}$ is the other.  Under the cube interpretation, these are the two infinite components that result if all the edges with color $i$ are deleted.  If we are examining the copy of the structure assigned to $\tau \in \omega^{<\omega}$, we have associated these two sets with a partition of the copy assigned to $\tau\cat i$: the former set is associated, via $P$, with the elements of even cardinality, while the latter set is associated with the elements of odd cardinality.

During the course of our construction, we will build an arithmetic tree $Q \subseteq \omega^{<\omega}$ which will be arithmetically isomorphic to $T$.  For $\sigma \in Q$, $S_n((F,\sigma))$ will hold for every $n$ and every $F \in [\omega]^{<\omega}$.  For $\sigma \not \in Q$, there will be an $n$ such that $S_n((F,\sigma))$ holds iff $F = \emptyset$.  The purpose of this setup is the following.

\begin{claim}\label{claim:muchnik}
The automorphisms of $\A$ which send $(\emptyset, \sigma)$ to $(\{i\}, \sigma)$ are Muchnik equivalent to $[\sigma\cat i]\cap [Q]$.  More generally, the automorphisms of $\A$ which send $(\emptyset,\sigma)$ to $(F, \sigma)$ are Muchnik equivalent to the join $\bigoplus_{i \in F} ([\sigma\cat i] \cap [Q])$.

In particular, such an automorphism exists iff the appropriate set is nonempty.
\end{claim}

\begin{proof}
Suppose $g$ is an automorphism of $\A$ sending $(\emptyset, \sigma)$ to $(F,\sigma)$, and $i \in F$.  We construct, effectively in $g$, a sequence $\sigma_0 \subset \sigma_1 \subset \sigma_2 \subset \dots$ with $\sigma_0 = \sigma\cat i$ and maintaining the inductive assumption that $g( (\emptyset, \sigma_j) ) \neq (\emptyset, \sigma_j)$.  By our earlier discussion about the $S_n$ and their relationship to $Q$, it will follow that each $\sigma_j \in Q$.

Note that $P( (\emptyset, \sigma), (\emptyset, \sigma\cat i))$, but $\neg P( (F, \sigma), (\emptyset, \sigma\cat i))$.  This shows that $g( (\emptyset, \sigma\cat i)) \neq (\emptyset, \sigma\cat i)$.  So $\sigma_0 = \sigma\cat i$ satisfies the inductive assumption.

Now, suppose we have defined $\sigma_j$.  Let $G$ be such that $g( (\emptyset, \sigma_j) ) = (G, \sigma_j)$.  By assumption, $G \neq \emptyset$, so fix $k \in G$.  As $P( (\emptyset, \sigma_j), (\emptyset, \sigma_j\cat k))$ but $\neg P( (G, \sigma_j), (\emptyset, \sigma_j\cat k))$, it must be that $g((\emptyset, \sigma_j\cat k)) \neq (\emptyset, \sigma_j\cat k)$.  So we let $\sigma_{j+1} = \sigma_j\cat k$.

\smallskip

Conversely, fix $(h_i)_{i \in F}$ with $h_i \in ([\sigma\cat i] \cap [Q])$.  We will construct, effectively in $(h_i)_{i \in F}$, an automorphism $g$ of $\A$ sending $(\emptyset, \sigma)$ to $(F,\sigma)$.  Define $g( (G, \sigma) ) = (G\triangle F, \sigma)$ for all $G \in [\omega]^{<\omega}$.  For each $\tau$ which is not an initial segment of any $h_i$, we define $g( (G, \tau) ) = (G,\tau)$.  For $\tau$ extending $\sigma$ and an initial segment of $h_i$, fix $k$ such that $\tau\cat k$ is an initial segment of $h_i$.  We define $g( (G, \tau)) = (G\triangle \{k\}, \tau)$ for all $G \in [\omega]^{<\omega}$.  For $\tau$ a proper initial segment of $\sigma$, our construction depends on the parity of $|F|$.  If $|F|$ is even, we define $g( (G, \tau)) = (G, \tau)$ for all $G \in [\omega]^{<\omega}$.  If $|F|$ is odd, we fix the $k$ such that $\tau\cat k$ is an initial segment of $\sigma$ and define $g( (G, \tau) ) = (G \triangle \{k\}, \tau)$ for all $G \in [\omega]^{<\omega}$.

As $g((G, \tau)) = (G, \tau)$ for any $\tau$ which is not an initial segment of any $h_i$, and thus for any $\tau \not \in Q$, we see that $g$ respects all the $S_n$.  As $g( (G,\tau) ) = (G\triangle H_\tau, \tau)$ for some $H_\tau$, for every $\tau$, we see that $g$ respects all the $E_i$ and $W_\tau$.  Finally, $|H_{\tau\cat k}|$ is odd precisely if $k \in H_\tau$, and so $g$ respects $P$.  Thus $g$ is an automorphism.
\end{proof}

\begin{claim}
$\A$ has non-computable Scott rank.
\end{claim}

\begin{proof}
It suffices to show that the orbits of $\A$ are not uniformly hyperarithmetic.  Note that $\sigma \in Q$ is extendible to a path in $[Q]$ precisely if there is an $i$ such that $(\emptyset, \sigma)$ and $(\{i\}, \sigma)$ are in the same orbit.  Thus if the orbits of $\A$ were uniformly hyperarithmetic, the set of extendible nodes in $Q$ would be hyperarithmetic.  But $Q$ is arithmetically isomorphic to $T$, the descending sequences through the Harrison order, and the set of extendible nodes of $T$ is not hyperarithmetic.
\end{proof}
In fact, one can make a similar argument to show that the orbit of $(\emptyset, \seq{})$ is not hyperarithmetic, and thus the Scott rank of $\A$ is $\wock+1$.\footnote{And even further, $\A$ has Scott complexity $\Pi_{\wock+2}$.}

\subsection*{Construction}
Fix $(\M_i)_{i \in \omega}$ a listing of all partial computable structures in our language.

We will perform a $\Pi^0_2$ priority construction on a tree of strategies.  We have strategies of type $N_\pi$ for $\pi \in T$, which are responsible for choosing the image of $\pi$ in $Q$ and ensuring that the $S_n$ are as described at this string.
We also have strategies of type $M_i$ for $i \in \omega$, which are responsible for constructing a computable isomorphism between $\A$ and $\M_i$ when $\A \cong \M_i$.  Finally, we have a global strategy $G$ which is responsible for ensuring that the $S_n$ are as described for those $\sigma \in \omega^{<\omega}$ not chosen by an $N_\pi$-strategy.

Each $N_\pi$-strategy will choose at most one string $\sigma$ to be the target for $\pi$.  If $\tau$ is an $N_\pi$-strategy which chooses $\sigma$ as the image of $\pi$, we say that $\tau$ has {\em chosen} $\sigma$.  No string $\sigma$ will be chosen by more than one strategy, and it will be the case that if $\sigma \in s^{<s}$ has not been chosen by the beginning of stage $s$, then it will never be chosen, and thus will not be in $Q$.  

We arrange the $N_\pi$ and $M_i$ in some priority ordering of order type $\omega$, such that for $\pi_0 \subset \pi_1 \in T$, $N_{\pi_0}$ occurs before $N_{\pi_1}$ in the ordering, and such that $N_{\seq{}}$ occurs first in the ordering.  The set of possible outcomes for an $M_i$-strategy is $\{ k, \infty_k : k \in \omega\} \cup \{\infty_\infty\}$.  We call the outcomes of the form $\infty_k$ for $k \in \omega \cup \{\infty\}$ the {\em infinite outcomes}.  Each $N_\pi$-strategy will have only a single outcome: $\outcome$.  We construct a priority tree of these strategies.

\smallskip

{\em Strategy for $G$.} At the end of stage $s$, for every $\sigma \in s^{<s}$ not yet chosen by any strategy, we declare $S_0( (\emptyset, \sigma) )$.

\smallskip

{\em Growing a string $\sigma$.}  At a stage $s$, a strategy may declare that $\sigma$ is growing.  When this happens, we let $n$ be largest such that we have already declared $S_n( (\emptyset, \sigma))$ (or $n = -1$ if there is no such $n$).  We declare $S_{k}( (\emptyset, \sigma))$ for $0 \le k < n+2$.  We also declare $S_k( (F, \sigma))$ for every $F \subseteq s$ and $0 \le k < n$.

\smallskip

{\em Strategy for $N_{\seq{}}$.}  This strategy is at the root of the priority tree and is visited at every stage.  It declares $\seq{}$ to be the image of $\seq{}$.  At every stage $s$, when this strategy is visited, the string $\seq{}$ grows.  This strategy always takes its unique outcome of $\outcome$.

\smallskip

{\em Strategy for $N_{\pi\cat i}$.}  Suppose $\tau$ is a strategy on the priority tree for requirement $N_{\pi\cat i}$.  Then there is a unique $\rho \subset \tau$ which is a strategy for $N_\pi$.  It has declared some unique string $\sigma$ to be the image of $\pi$.  Let $s_0$ be the first stage at which $\tau$ is visited.  At this stage, we choose an $m > s_0$ which has not yet been mentioned in the construction and declare $\sigma\cat m$ to be the image of $\pi\cat i$.  

At any stage when $\tau$ is visited, $\sigma\cat m$ grows.

At every stage $\tau$ is visited, it takes its unique outcome of $\outcome$.

\smallskip

{\em An auxiliary definition.} For every $\sigma$ and $s$ with $\sigma \in (s-1)^{<s-1}$, we define $n_\sigma(s)$ to be the largest $n$ such that we have declared $S_n( (\emptyset,\sigma))$ by the start of stage $s$.  By our action for $G$ and the $N_\pi$-strategies, $n_\sigma(s)$ is always defined.

\smallskip

{\em Strategy for $M_i$.}  Suppose $\tau$ is a strategy on the priority tree for requirement $M_i$.  Let $C_\tau$ be those strings $\sigma$ which are chosen by some $N_\pi$-strategy $\rho$ with $\rho \subset \tau$.  It will be the case that $C_\tau$ is fully determined by the first stage at which $\tau$ is visited.  We construct a partial computable function $f_\tau: \omega^{<\omega} \to \M_i$.  We will also define sequences $x_\tau(\sigma, s)$ for $\sigma \in C_\tau$.  If $\M_i \cong \A$ and $\tau$ is on the true path, we will use $f_\tau$ and these sequences to post hoc construct an isomorphism between $\A$ and $\M_i$.

When $\tau$ is visited at stage $s$, let $k_0$ be the number of times $\tau$ has previously taken an infinite outcome, and let $t<s$ be the last stage at which $\tau$ took an infinite outcome ($t = 0$ if there is no such stage).  Let $k_1$ be the number of times $\tau$ has previously taken outcome $\infty_\infty$.  Let $B_\tau(s)$ consist of every string $\sigma \in t^{<t}$ except for those which have been chosen by some $N_\pi$-strategy $\rho$ with $\rho \supseteq \tau\cat k_0$ or $\rho \subset \tau$.  For each $\sigma \in B_\tau(s)$, if $f_\tau( \sigma )$ has not yet been defined, we search for $s$ steps for an element $x \in \M_i$ with $W_\sigma(x)$ and $S_{n_\sigma(t+1)}(x)$.  If we find such an element, we define $f_\tau( \sigma ) = x$.

If there is at least one $\sigma \in B_\tau(s)$ for which at least one of the following fails, then we finish our action for $\tau$ at stage $s$ and take outcome $k_0$:
\begin{itemize}
\item $f_\tau( \sigma)$ is defined;
\item $\M_i \models S_{n_\sigma(t+1)}( f_\tau(\sigma))$; and
\item For every $j$ with $\sigma\cat j \in B_\tau(s)$, $\M_i \models P( f_\tau(\sigma), f_\tau(\sigma\cat j))$,
\end{itemize}

If the above holds for every $\sigma \in B_\tau(s)$, we will have an infinite outcome at stage $s$, but it remains to determine which.  First, let $D_\tau(s)$ consist of those strings $\sigma \in B_\tau(s)$ which have not been chosen by a strategy extending $\tau\cat \infty_\infty$.  For each $\sigma \in C_\tau$, we let $x_\tau(\sigma, s)$ be the oldest element $x$ of $\M_i$ such that $\M_i \models W_\sigma(x) \wedge P(x, f_\tau(\sigma\cat j))$ for every $\sigma\cat j \in D_\tau(s)$, if we see such an $x$ at stage $s$, and we leave $x_\tau(\sigma, s)$ undefined if there is no such $x$.

If, for every $\sigma \in C_\tau$, both $x_\tau(\sigma, s)$ and $x_\tau(\sigma, t)$ are defined, and $x_\tau(\sigma,s) = x_\tau(\sigma, t)$, then we take outcome $\infty_{k_1}$.  Otherwise, we take outcome $\infty_\infty$.

\smallskip

{\em Running the construction.}  At stage $s$, we begin by visiting the strategy $\seq{}$.  After having visited a strategy $\tau$ at stage $s$ with $|\tau| < s$, we let $k$ be the outcome taken by $\tau$ at this stage.  We next visit the strategy $\tau\cat k$.   After having visited a strategy $\tau$ with $|\tau| = s$, we run the global strategy and then end the stage.

We put the lexicographic ordering on the tree of strategies, where we order the outcomes of $M_i$ strategies as $\infty_\infty < \dots < \infty_2 < \infty_1 < \infty_0 < \dots < 2 < 1 < 0$.  Observe that our construction has the property that if a strategy $\tau$ is visited, and then at a later stage a strategy $\rho$ which is lexicographically to the left of $\tau$ is visited, $\tau$ can never again be visited.

\subsection*{The true path} We define the true path recursively, maintaining the inductive hypothesis that every strategy on the true path is visited infinitely often.  $\seq{}$ is on the true path.  If $\tau$ is on the true path and $\tau$ is an $N_\pi$-strategy, then $\tau\cat\outcome$ is on the true path.

If $\tau$ is on the true path and $\tau$ is an $M_i$-strategy, let $\alpha$ be the leftmost outcome which $\tau$ takes infinitely often.  Observe that this necessarily exists by our action for $\tau$.  Then $\tau\cat \alpha$ is on the true path.  We call $\alpha$ the {\em true outcome} of $\tau$.

By the earlier observation, no strategy off the true path is visited infinitely often.  Observe that the true path is computable from $\emptyset''$.

\subsection*{Verification} We begin by defining $Q$ and an isomorphism $\phi: T \to Q$.  For $\pi \in T$, let $\tau$ be the unique $N_\pi$-strategy along the true path.  We define $\phi(\pi)$ to be the string which $\tau$ declared to be the image of $\pi$.  As the true path is $\Delta^0_3$, $\phi$ is arithmetic.  We then define $Q$ to be the image of $T$ under $\phi$.  It is a straightforward induction to show that $Q$ is a tree and $\phi$ is an isomorphism.

\begin{claim}
If $\sigma \in Q$, then $\A \models S_n(F, \sigma)$ for every $F \in [\omega]^{<\omega}$.

If $\sigma \not \in Q$, then $n_\sigma = \lim_s n_\sigma(s)$ exists, and $\A \models S_{n_\sigma}(F, \sigma)$ iff $F = \emptyset$.
\end{claim}

\begin{proof}
If $\sigma \in Q$, then $\sigma$ was selected by some $N_\pi$-strategy $\tau$ along the true path.  As $\tau$ is visited infinitely often, our action for $\tau$ ensures that $\A \models S_n(F, \sigma)$ for every $F \in [\omega]^{<\omega}$.

If $\sigma \not \in Q$, then either $\sigma$ was never chosen by an $N_\pi$-strategy, or it was chosen by one off the true path.  In the former case, $n_\sigma = 0$, and we never declared $S_0((F, \sigma))$ for any $F \neq \emptyset$.  In the latter case, the strategy was only visited finitely often, so $n_\sigma$ exists.  By construction, we have not declared $S_{n_\sigma}((F, \sigma))$ for any $F \neq \emptyset$.
\end{proof}

Next, we must show that $\A$ is computably categorical.

\begin{claim}
Fix $\tau$ an $M_i$-strategy.

If $\tau$ is visited at stages $s_0 < s_1$, then $B_\tau(s_0) \subseteq B_\tau(s_1)$.

Further, if there is no stage $t \in [s_0, s_1)$ at which $\tau$ takes an infinite outcome, then $B_\tau(s_0) = B_\tau(s_1)$.
\end{claim}

\begin{proof}
For $i < 2$, let $k_{i}$ be the number of times $\tau$ took an infinite outcome before stage $s_i$, and let $t_i$ be the last stage before $s_i$ at which $\tau$ took an infinite outcome ($t_i = 0$ if there is no such stage).  By definition, if there were some $\sigma \in B_\tau(s_0) - B_\tau(s_1)$, then it would have to be the case that $\sigma \in t_0^{<t_0}$ (and so $t_0 > 0$), $\sigma$ was selected by some strategy extending $\tau\cat k_1$, and $k_0 \neq k_1$.  But no strategy extending $\tau\cat k_1$ is accessible before stage $s_0$, and $\sigma$ cannot be selected after stage $t_0$, and $t_0 < s_0$.

By definition, if there is $\sigma \in B_\tau(s_1) - B_\tau(s_0)$ then either $t_1 > t_0$ or $\sigma$ was selected by a strategy extending $\tau\cat k_0$ which does not extend $\tau\cat k_1$, and thus $k_0 \neq k_1$ (which again implies $t_1 > t_0$).  In either case, we see that there is some stage $t \in [s_0, s_1)$ at which $\tau$ took an infinite outcome.
\end{proof}

\begin{claim}
Fix $\tau$ the $M_i$-strategy along the true path.  If $\tau$'s true outcome is some $k \in \omega$, then $\A \not \cong \M_i$.
\end{claim}

\begin{proof}
Note that $k$ is the number of times $\tau$ ever takes an infinite outcome.  Let $t$ be the last stage at which $\tau$ takes an infinite outcome (or $t = 0$ if there is no such stage).  Let $\alpha$ be the infinite outcome taken at stage $t$, if there is one.  Let $s_0$ be the first stage after $t$ at which $\tau$ is visited.  Then for every stage $s > s_0$ at which $\tau$ is visited, $B_\tau(s) = B_\tau(s_0)$.

Observe that for every $\sigma \in B_\tau(s_0)$, and every $s > s_0$, $n_\sigma(s) = n_\sigma(s_0)$.  For if $\sigma \in t^{<t}$ was not chosen by any strategy by stage $t$, then it will never be chosen, and our action for $G$ makes $n_\sigma(s) = 0$ for every $s > t$.  If instead $\sigma$ is chosen by a strategy by stage $t$, it must be a strategy which is incomparable with $\tau\cat k$ by the definition of $B_\tau(s_0)$.  Such a strategy can never be visited at or after stage $s_0$, and so $n_\sigma(s)$ will not change after $s_0$.

For $\sigma \in B_\tau(s_0)$, we write $n_\sigma$ for $\lim_s n_\sigma(s) = n_\sigma(s_0)$.  By construction, for every $\sigma \in B_\tau(s_0)$, and every $F \neq \emptyset$, $\A \models \neg S_{n_\sigma}((F,\sigma)) \wedge \neg S_{n_\sigma -1}((F,\sigma))$.  There is some $\sigma \in B_\tau(s_0)$ which prevents $\tau$ from taking an infinite outcome after stage $t$.  We consider the various cases.

If $f_\tau(\sigma)$ is never defined, then there is no $x \in \M_i$ with $W_\sigma(x)$ and $S_{n_\sigma}(x)$.  On the other hand, $\A \models W_\sigma((\emptyset, \sigma)) \wedge S_{n_\sigma}((\emptyset, \sigma))$, so $\A \not \cong \M_i$.

If $f_\tau(\sigma)$ is defined but $\M_i \models \neg S_{n_\sigma}(f_\tau(\sigma))$, then it must be that $f_\tau(\sigma)$ was already defined by stage $t$, and thus $\sigma \in B_\tau(t)$.  For if not, then $f_\tau(\sigma)$ was defined at some stage $s \ge s_0$, and by construction we would have $\M_i \models S_{n_\sigma}(f_\tau(\sigma))$.  By the fact that $f_\tau(\sigma)$ is defined, we know that $\M_i \models W_\sigma(f_\tau(\sigma))$.  As $B_\tau(0)$ is empty, we know that $t > 0$ and thus $\tau$ had infinite outcome at stage $t$.

$\M_i \models S_{n_\sigma(t)}(f_\tau(\sigma))$ by the fact that $\tau$ had an infinite outcome at stage $t$.  Since $n_\sigma \neq n_\sigma(t)$, $\sigma$ must have been chosen by a strategy $\rho$ which is visited in the interval $[t,s)$.  By the fact that $\sigma \in B_\tau(t) \subseteq t^{<t}$, we know that $\rho$ chose $\sigma$, and thus was visited, at or before stage $t$.  Since it is also visited at or after stage $t$, and $\tau$ has infinite outcome $\alpha$ at stage $t$, $\rho$ must be comparable with $\tau\cat \alpha$.  By definition of $B_\tau(t)$, $\rho \not \subseteq \tau$.  Thus $\rho \supseteq \tau\cat\alpha$.  So $\rho$ is visited at stage $t$ and at no stage after $t$.

By our action for $G$, it follows that $n_\sigma = n_\sigma(t)+1$.  Further, the only element $y \in \A$ with $\A \models W_\sigma(y) \wedge S_{n_\sigma(t)}(y)$ is $(\emptyset, \sigma)$, and $\A \models S_{n_\sigma}( (\emptyset, \sigma))$.  So there is no element $y \in \A$ with $\A \models W_\sigma(y) \wedge S_{n_\sigma(t)}(y) \wedge \neg S_{n_\sigma}(y)$, and so $\A \not \cong \M_i$.

Finally, suppose there is $\sigma\cat j \in B_\tau(s_0)$ with $f_\tau(\sigma), f_\tau(\sigma\cat j)$ both defined, and
\[
\M_i \models W_\sigma(f_\tau(\sigma)) \wedge W_{\sigma\cat j}(f_\tau(\sigma\cat j)) \wedge S_{n_\sigma}(f_\tau(\sigma)) \wedge S_{n_{\sigma\cat j}}(f_\tau(\sigma\cat j)) \wedge \neg P( f_\tau(\sigma), f_\tau(\sigma\cat j)).
\]
As $\sigma\cat j \in B_\tau(s_0)$, $n_{\sigma\cat j} = n_{\sigma\cat j}(s_0)$.  So $(\emptyset, \sigma)$ and $(\emptyset,\sigma\cat j)$ are the only $y_0, y_1 \in \A$ satisfying
\[
\A \models W_\sigma(y_0) \wedge W_{\sigma\cat j}(y_1) \wedge S_{n_\sigma}(y_0) \wedge S_{n_{\sigma\cat j}}(y_1),
\]
and $\A \models P((\emptyset, \sigma), (\emptyset, \sigma\cat j))$.  Thus $\A \not \cong \M_i$.
\end{proof}

\begin{claim}
Fix $\tau$ the $M_i$-strategy along the true path.  If $h: \M_i \cong \A$, then for all $\sigma \not \in Q$, $h\circ f_\tau(\sigma) = (\emptyset, \sigma)$.
\end{claim}

\begin{proof}
As $\M_i \cong \A$, $\tau$'s true outcome must be an infinite outcome.  Thus every $\sigma \not \in C_\tau$ is in the domain of $f_\tau$, and in particular every $\sigma \not \in Q$ is in the domain of $f_\tau$.

Since $\tau$ takes infinite outcomes at infinitely many stages, for every $\sigma \not \in Q$, $\M_i \models W_\sigma(f_\tau(\sigma)) \wedge S_{n_\sigma}(f_\tau(\sigma))$.  For $\sigma \not \in Q$, $(\emptyset, \sigma)$ is the only element of $\A$ satisfying both $W_\sigma$ and $S_{n_\sigma}$, so for $h$ to be an isomorphism, it must be that $h\circ f_\tau(\sigma) = (\emptyset, \sigma)$.
\end{proof}

\begin{claim}
Fix $\tau$ the $M_i$-strategy along the true path.  If $\tau$'s true outcome is $\infty_\infty$, then $\A \not \cong \M_i$.
\end{claim}

\begin{proof}
Let $t_0 < t_1 < \dots$ be the stages at which $\tau$ takes an infinite outcome.  Then it must be that for some $\sigma \in C_\tau$, $\lim_s x_\tau(\sigma, t_s)$ does not exist.  Fix such a $\sigma$.  Note that if $s_0 < s_1$ and $x_\tau(\sigma, t_{s_0}), x_\tau(\sigma, t_{s_1})$ are both defined but unequal, then $x_\tau(\sigma, t_{s_0}) < x_\tau(\sigma,t_{s_1})$.  This is because $x_\tau(\sigma, t_{s_1})$ is an $x$ of the sort sought for $x_\tau(\sigma, t_{s_0})$, and if it were older it would have been chosen at stage $t_{s_0}$.  

Let $D_\tau = \bigcup_s D_\tau(t_s)$.  There is no $x \in \M_i$ such that $\M_i \models P(x, f_\tau(\sigma\cat j))$ for every $\sigma\cat j \in D_\tau$ -- if there were, $x_\tau(\sigma, t_s)$ would converge to the least such.  Observe that $D_\tau \cap Q = \emptyset$, for if a string is in $Q$, then it is selected by a strategy which is an initial segment of $\tau$ or which extends $\tau\cat \infty_\infty$, and these strings are excluded from $D_\tau$.  

If there were an isomorphism $h: \M_i \cong \A$, then $h\circ f_\tau(\sigma\cat j) = (\emptyset, \sigma\cat j)$ for every $\sigma\cat j \not \in Q$, and in particular every $\sigma\cat j \in D_\tau$, and $P( (\emptyset, \sigma), (\emptyset, \sigma\cat j))$ for every $j$.  So $x = h^{-1}((\emptyset, \sigma))$ would satisfy $\M_i \models P(x, f_\tau(\sigma\cat j))$ for every $\sigma\cat j \in D_\tau$, contradicting the previous paragraph.
\end{proof}

\begin{claim}
If $\A \cong \M_i$, then there is a computable isomorphism between $\A$ and $\M_i$.
\end{claim}

\begin{proof}
Fix $\tau$ the $M_i$-strategy along the true path.  Then $\tau$'s true outcome is $\infty_k$ for some $k \in \omega$, and so the domain of $f_\tau$ is all strings $\sigma \in \omega^{<\omega}$ except those chosen by some strategy $\rho \subset \tau$, which is to say $\dom f_\tau = \omega^{<\omega} - C_\tau$.  Note that the domain of $f_\tau$ is thus a cofinite, upwards-closed set.  By assumption, there is an isomorphism $h: \M_i \to \A$.  We begin by non-uniformly extending $f_\tau$ to all of $\omega^{<\omega}$.  We do so recursively along $C_\tau$.

For $\sigma \in C_\tau$, assume $f_\tau(\sigma\cat j)$ is already defined for every $j \in \omega$.  Let $t_0 < t_1 < \dots$ be the stages at which $\tau$ has an infinite outcome.  As $\tau$'s true outcome is $\infty_k$, $x = \lim_s x_\tau(\sigma, t_s)$ exists.  So for every $\sigma\cat j \in D_\tau = \bigcup_s D_\tau(t_s)$, $\M_i \models P(x, f_\tau(\sigma\cat j))$.  Since $C_\tau$ is finite and $\tau\cat \infty_\infty$ is visited only finitely many times, there are only finitely many $j$ with $\sigma\cat j \not \in D_\tau$.  Let $J = \{ j : \M_i \models \neg P(x, f_\tau(\sigma\cat j))\}$, and note that $J$ is finite.  Informally, we let $f_\tau(\sigma) = x\triangle J$.  More precisely, if $h(x) = (F, \sigma)$, we let $f_\tau(\sigma) = h^{-1}( (F \triangle J, \sigma))$.

Having extended $f_\tau$ to a total computable function (albeit non-uniformly), let $H_\sigma$ be such that $h(f_\tau(\sigma)) = (H_\sigma, \sigma)$, for each $\sigma \in \omega^{<\omega}$.

\begin{subclaim}
For all $\sigma$ and $j$, $\M_i \models P( f_\tau(\sigma), f_\tau(\sigma\cat j))$.
\end{subclaim}

\begin{proof}
If $\sigma \not \in C_\tau$, then $\sigma\cat j \not \in C_\tau$, and so $f_\tau(\sigma), f_\tau(\sigma\cat j)$ were defined by our action for $\tau$, and since $\tau$'s true outcome is an infinite outcome, $\M_i \models P( f_\tau(\sigma), f_\tau(\sigma\cat j))$.

If instead $\sigma \in C_\tau$, then this is by construction.  Consider the $x = \lim_s x_\tau(\sigma, t_s)$ and the set $J$ used in the definition of $f_\tau(\sigma)$.  If $j \not \in J$, then $\M_i \models P(x, f_\tau(\sigma\cat j))$.  So $\A \models P(h(x), (H_{\sigma\cat j}, \sigma\cat j))$.  By construction, $h(x) = (H_\sigma\triangle J, \sigma)$, and since $j \not \in J$, $j \in H_\sigma$ iff $j \in H_\sigma\triangle J$, so $\A \models P( (H_\sigma, \sigma), (H_{\sigma\cat j}, \sigma\cat j))$, and thus $\M_i \models P( f_\tau(\sigma), f_\tau(\sigma\cat j))$.

If instead $j \in J$, then $\M_i \models \neg P(x, f_\tau(\sigma\cat j))$.  So $\A \models \neg P(h(x), (H_{\sigma\cat j}, \sigma\cat j))$.  Since $j \in J$, $\A \models P( (H_\sigma, \sigma), (H_{\sigma\cat j}, \sigma\cat j))$, and thus $\M_i \models P( f_\tau(\sigma), f_\tau(\sigma\cat j))$.
\end{proof}

It follows that for every $\sigma$ and $j$, $j \in H_\sigma$ iff $|H_{\sigma\cat j}|$ is odd.

We now define a computable isomorphism $g: \A \to \M_i$.  We begin by defining $g( (\emptyset, \sigma) ) = f_\tau(\sigma)$ for every $\sigma \in \omega^{<\omega}$.  We then extend recursively.  For each $F \in [\omega]^{<\omega}$ with $|F| > 0$, we search for $G$, $j$ and $y$ with $F = G \cup \{j\}$, $g( (G, \sigma))$ already defined, and $\M_i \models E_j( g( (G, \sigma)), y)$, and define $g( (F, \sigma) ) = y$ for the first such found.

To show that $g$ is an isomorphism, we will show that $h\circ g$ is an automorphism of $\A$.

\begin{subclaim}
For each $(F, \sigma) \in \A$, $h\circ g( (F, \sigma)) = (F\triangle H_\sigma, \sigma)$.
\end{subclaim}

\begin{proof}
We argue by induction on $|F|$.  The base case $|F| = 0$ is immediate.

For $|F| > 0$, there was some $j \in F$ such that $g( (F, \sigma))$ was defined based on $G = F - \{j\}$.  Then $\M_i \models E_j(g( (G, \sigma)), g((F, \sigma)))$, and so
\begin{align*}
\A \models E_j( h\circ g( (G, \sigma)), h\circ g( (F, \sigma))) &\implies
h\circ g( (G, \sigma)) \triangle h\circ g( (F, \sigma)) = \{j\}\\
&\implies (G\triangle H_\sigma)\triangle h\circ g( (F, \sigma)) = \{j\}\\
&\implies (F\triangle \{j\})\triangle H_\sigma \triangle h\circ g( (F, \sigma)) = \{j\}\\
&\implies h\circ g( (F, \sigma)) = F\triangle H_\sigma.\qedhere
\end{align*}
\end{proof}

\begin{subclaim}
$h\circ g$ is an automorphism of $\A$.
\end{subclaim}

\begin{proof}
It follows from the previous claim that $h\circ g$ respects all the relations $E_j$.

As $j \in H_\sigma$ iff $|H_{\sigma\cat j}|$ is odd, $h\circ g$ respects the $P$ relation.

If $\sigma \in Q$, then $S_n( (F, \sigma))$ for every $F \in [\omega]^{<\omega}$ and $n \in \omega$, and so $h\circ g$ trivially respects the $S_n$ relations on elements of the form $(F, \sigma)$.

If $\sigma \not \in Q$, then by construction, $\M_i \models W_\sigma(f_\tau(\sigma)) \wedge S_{n_\sigma}(f_\tau(\sigma))$.  Further, $(\emptyset, \sigma)$ is the only element of $\A$ satisfying this formula, so it must be that $h\circ f_\tau(\sigma) = (\emptyset, \sigma)$.  Thus $H_\sigma = \emptyset$, and so $h\circ g$ fixes elements of the form $(F, \sigma)$, and thus respects the $S_n$ relations on them.
\end{proof}

So $h\circ g$ is an automorphism of $\A$, and thus $g$ witnesses that $\A$ and $\M_i$ are computably isomorphic.
\end{proof}
This completes the proof of \Cref{thm:cc_high_Scott}.

\smallskip

The computably categorical structure $\A$ constructed in the proof of \Cref{thm:cc_high_Scott} has the property that it is no longer computably categorical after a constant is added naming $(\emptyset, \seq{})$.  For fix any $i$ such that $[i] \cap [Q] \neq \emptyset$.  Then, as shown in \cref{claim:muchnik}, $(\A, (\emptyset, \seq{}) ) \cong (\A, (\{i\}, \seq{}))$, but computing an isomorphism requires computing a path through $Q$, and $Q$ has no hyperarithmetic paths.

Indeed, for any distinct $i_0, i_1$, there is no computable isomorphism between $(\A, (\{i_0\}, \seq{}))$ and $(\A, (\{i_1\}, \seq{}))$.  Thus $(\A, (\emptyset, \seq{}))$ has infinite computable dimension.  A structure with this behavior was first constructed by Hirschfeldt, Khoussainov and Shore~\cite{HiKhSh03}.

\begin{corollary}
There is a structure $\B_0$ of computable dimension 2 with two computable copies which are not hyperarithmetically isomorphic.
\end{corollary}

\begin{proof}
Let $\A$ be the structure from the proof of \Cref{thm:cc_high_Scott}.  Our structure $\B_0$ will consist of $\A$ along with two new elements, $a_{\text{even}}$ and $a_{\text{odd}}$, and a new constant symbol $c$.  We make the following additional definitions:
\begin{itemize}
\item $\B_0 \models P( a_\text{even}, (F, \seq{}))$ precisely when $|F|$ is even;
\item $\B_0 \models P( a_\text{odd}, (F, \seq{}))$ precisely when $|F|$ is odd;
\item $\B_0 \models c = a_\text{even}$;
\end{itemize}
and no other relations hold involving $a_\text{even}$ or $a_\text{odd}$.

Let $\B_1$ be defined identically to $\B_0$, excepting that $\B_1 \models c = a_\text{odd}$.  Let $\hat{\B}$ be the structure obtained from $\B_0$ (equivalently, from $\B_1$) by removing the constant symbol $c$ from the language.

For any computable $\hat{\C} \cong \hat{\B}$, if we nonuniformly remove the copies of $a_{\text{even}}$ and $a_{\text{odd}}$, then we have a computable copy of $\A$.  As $\A$ is computably categorical, there is a computable isomorphism $f$ from this smaller structure to $\A$, and this can be extended to a computable isomorphism from $\hat{\C}$ to $\hat{\B}$.

So for any $\C \cong \B_0$, the computable isomorphism from $\hat{\C}$ to $\hat{B}$ is either an isomorphism from $\C$ to $\B_0$ or from $\C$ to $\B_1$.  However, any isomorphism from $\B_0$ to $\B_1$ must restrict to an automorphism of $\A$ sending $(\emptyset, \seq{})$ to some $(F, \seq{})$ with $|F|$ odd, and by \Cref{claim:muchnik} and choice of computable tree $T$, such an isomorphism must not be hyperarithmetic.
\end{proof}

We are grateful to Rod Downey for asking whether our methods could address the previous question.

\section{A structure with degree of categoricity and infinite spectral dimension}

This section is devoted to the proof of the following theorem.

\begin{theorem}\label{thm:weak_deg_cat}
There is a computable structure with degree of categoricity $\zerojj$, but with infinite spectral dimension.
\end{theorem}

We will first build a computable structure $\A$ and closed sets $Q, R \subseteq \omega^\omega$ satisfying all of the following:
\begin{itemize}
\item $\A$ is computably categorical.
\item $X = \{ i : [i] \cap Q \neq \emptyset\}$ is infinite.
\item For all $i \in X$, $|[i] \cap Q| = 1$.  We refer to the unique element of $[i] \cap Q$ as $f_i$.
\item $Y = \{ j : [j] \cap R \neq \emptyset\}$ is infinite.
\item For all $j \in Y$, $|[j] \cap R| = 1$.  We refer to the unique element of $[j] \cap R$ as $g_j$.
\item Each $f_i$ and $g_j$ is $\Delta^0_3$.
\item For any $i \in X$ and any $j \in Y$ with $i < j$, $f_i + g_j$ computes $\emptyset''$.
\item For any $j_0, \dots, j_k \in Y$ and any $i \in X$ with $i > \max\{j_0, \dots, j_k\}$, $f_i \oplus g_{j_0} \oplus \dots \oplus g_{j_k}$ does not compute $\emptyset'$.
\item There are elements $\{u_0, u_1\}$ which comprise an orbit of $\A$.  Computing an automorphism sending $u_0$ to $u_1$ exactly corresponds to computing an element of $Q$ (i.e.\ the tasks are Muchnik-equivalent).
\item There are elements $\{v_F : F \in [Y]^{<\omega}\}$ which comprise an orbit of $\A$.  Computing an automorphism sending $v_F$ to $v_G$ (and fixing $u_0$) exactly corresponds to computing $\displaystyle \bigoplus_{j \in F\triangle G} g_j$.
\end{itemize}

Before we proceed with the construction, we explain why this will suffice.  Our final structure will be $(\A, c, d)$, where $c$ and $d$ are constants naming $u_0$ and $v_\emptyset$.  Suppose $\B$ is another computable presentation of $(\A, c, d)$.  As $\A$ is computably categorical, we may assume that $\B$ has the form $(\A, u_\ell, v_F)$ for $\ell \in \{0, 1\}$, $F \in [Y]^{<\omega}$.  Then, for any $i \in X$, $\displaystyle f_i \oplus \bigoplus_{j \in F} g_j$ computes an isomorphism from $(\A, c, d)$ to $\B$, and in particular $\zerojj$ does.

If $\B_0, \dots, \B_{n-1}$ are all computable presentations of $(\A, c, d)$, then we may assume that $\B_k$ is of the form $(\A, u_{\ell_k}, v_{F_k})$.  Then, for any $i \in X$, $\displaystyle f_i \oplus \bigoplus_{j \in \bigcup_{k < n} F_k} g_j$ computes isomorphisms between all of these structures, and if $i$ is chosen sufficiently largely, it does not compute $\emptyset'$.

On the other hand, suppose $\degd$ can compute an isomorphism between any two computable presentations of $(\A, c, d)$.  Then it computes an isomorphism between $(\A, c, d)$ and $(\A, u_1, d)$, and thus must compute some $f_i \in Q$.  Fix $j > i$.  Then $\degd$ must also compute an isomorphism between $(\A, c, d)$ and $(\A, c, \{j\})$, and thus must compute $g_j$.  Since $f_i + g_j$ computes $\emptyset''$, $\degd \ge \zerojj$.

\subsection*{The structure}
Our structure will have $(W_\sigma^Q)_{\sigma \in \omega^{<\omega}}, (W_\sigma^R)_{\sigma \in \omega^{<\omega}}$ all unary relations, $(E_i)_{i \in \omega}$ all binary relations, $P$ a binary relation, and $(S_n)_{n \in \omega}$ all unary relations.  As in the previous construction, the $S_n$ will only be declared in a c.e.\ fashion.  Our structure's universe will be $\left([\omega]^{<\omega} \times \omega^{<\omega}\times\{0,1\}\right) \sqcup\{u_0, u_1\}$.

With the exception of the $S_n$, we can give a complete description of our structure immediately.  For $(F, \tau, a), (G, \rho, b) \in [\omega]^{<\omega} \times \omega^{<\omega}\times \{0,1\}$:
\begin{itemize}
\item $W_\sigma^Q( (F, \tau, a) )$ holds iff $\sigma = \tau$ and $a = 0$.
\item $W_\sigma^R( (F, \tau, a) )$ holds iff $\sigma = \tau$ and $a = 1$.
\item $E_i( (F,\tau,a), (G,\rho, b) )$ holds iff $\tau = \rho$, $a = b$ and $F\triangle G = \{i\}$.
\item $P( (F,\tau, a), (G,\rho, b) )$ holds iff $a = b$, $\rho = \tau\cat i$ for some $i$, and one of the following holds:
\begin{itemize}
\item $i \not \in F$ and $|G|$ is even; or
\item $i \in F$ and $|G|$ is odd.
\end{itemize}
\item For $k < 2$, $P( u_k, (F,\seq{}, 0))$ holds iff $k + |F|$ is even.
\item No other relations involving $u_0$ or $u_1$ hold.
\end{itemize}
As before, we have the affine space $[\omega]^{<\omega}$, but this time we have two copies of $\omega^{<\omega}$, and we have associated a copy of the space to each string in either copy.  The $W_\sigma^Q$ and $W_\sigma^R$ let us identify which copy a given element belongs to.

As mentioned, during the course of the construction we will build closed sets $Q, R \subseteq \omega^{\omega}$.  Let $T_0, T_1 \subseteq \omega^{<\omega}$ be the minimal trees such that $[T_0] = Q$, $[T_1] = R$, i.e.\ $T_0 = \{ \sigma \in \omega^{<\omega} : [\sigma] \cap Q \neq \emptyset\}$, and similarly for $T_1$.  For $a < 2$ and $\sigma \in T_a$, $S_n(F, \sigma, a)$ will hold for every $n$ and every $F \in [\omega]^{<\omega}$.  For $\sigma \not \in T_a$, there will be an $n$ such that $S_n(F, \sigma, a)$ holds iff $F = \emptyset$.  The purpose of this setup is the following.

\begin{claim}\label{claim:muchnik_again}
The automorphisms of $\A$ which send $(\emptyset, \sigma, 1)$ to $(\{i\}, \sigma, 1)$ are Muchnik equivalent to $[\sigma\cat i]\cap R$.  More generally, the automorphisms of $\A$ which send $(\emptyset,\sigma, 1)$ to $(F, \sigma, 1)$ are Muchnik equivalent to the join $\bigoplus_{i \in F} ([\sigma\cat i] \cap R)$.

The automorphisms which send $u_0$ to $u_1$ are Muchnik equivalent to $Q$.

In particular, such an automorphism exists iff the appropriate set is nonempty.
\end{claim}

\begin{proof}
As the proof of \cref{claim:muchnik}.
\end{proof}

Let $v_F = (F, \seq{}, 1)$ for $F \in [Y]^{<\omega}$.  These, along with $u_0, u_1$, are the previously promised elements.

\subsection*{Construction}
Fix $(\M_i)_{i \in \omega}$ a listing of all partial computable structures in our language.  Fix also $Z$ a set which is complete for $\Pi^0_2$, and with the property that any enumeration of $Z$ computes $Z$.  For example, $Z = \overline{\emptyset'}\oplus \overline{\emptyset''}$.  Fix a computable predicate $\phi(n,s)$ such that $n \in Z \iff \exists^\infty s\, \phi(n,s)$.  For $i < j$, our intention is that $f_i + g_j$ is a ``modulus for $\Sigma^0_2$-ness'':
\begin{center}
for $n > j$, $\exists^\infty s\, \phi(n, s) \iff \exists s > (f_i+g_j)(n)\, \phi(n,s)$.
\end{center}
This will ensure that $f_i + g_j$ enumerates $Z$ and thus computes $Z$.

Again, we perform a $\Pi^0_2$ priority construction on a tree of strategies.  We have strategies of type $N^a_r$, for $a < 2$ and $r \in \omega$, which are mother strategies which begin an $f_i$ or $g_i$ with $i > r$, depending on whether $a = 0$ or $1$.  We have strategies of type $N^a_{r,n}$ for $a < 2$, $r \in \omega$ and $n > 0$, which are daughter strategies responsible for defining $f_i(n)$ or $g_j(n)$ and ensuring that $\exists^\infty s\, \phi(n, s) \iff \exists s > (f_i+g_j)(n)\, \phi(n,s)$, where either $f_i$ or $g_j$ is the path begun by the mother node.  We have strategies of type $U_{i,e}$, for $i,e \in \omega$, which are responsible for ensuring that $f_i \oplus g_0 \oplus \dots \oplus g_{i-1}$ does not compute $\emptyset'$ via functional $\Phi_e$.  We have strategies of type $M_i$, for $i \in \omega$, which are responsible for constructing a computable isomorphism between $\A$ and $\M_i$ when $\A \cong \M_i$.  Finally, we have a global strategy $G$ which is responsible for ensuring that the $S_n$ are as described.

We arrange the $N_r^a$, $N_{r,n}^a$, $U_{i,e}$ and $M_i$ in some priority ordering of type $\omega$, such that $N_r^a$ occurs in the ordering before any $N_{r,n}^a$, and each $N_{r,n}^a$ occurs before $N_{r,n+1}^a$.  The set of possible outcomes for an $M_i$-strategy is again $\{k, \infty_k : k \in \omega\} \cup \{\infty_\infty\}$.  $N_r^a$-strategies have only a single outcome: $\outcome$.  The set of possible outcomes for an $N_{r,n}^a$-strategy is $\{k : k \in \omega\} \cup \{\infty\}$.  The set of possible outcomes for a $U_{i,e}$-strategy is $\{0, 1\}$.

We will construct a priority tree of these strategies, but this tree will not be the simplest tree, where each level is devoted to a given strategy type according to the priority ordering.  Instead, we will need to dynamically construct our tree during the construction.  We will say more about this later.

\smallskip

{\em Auxiliary functions.}  During the construction, we will be defining various $f_i$ and $g_j$.  Each will be begun by a mother node of some $N_r^a$-type, and then will be continued by nodes of $N_{r, n}^a$-type but also by nodes of $U_{i,e}$-type.  The $N_{r,n}^a$-nodes will inherit a string from an ancestor strategy which they will then extend, while the $U_{i,e}$-nodes may steal strings from descendant strategies.  To track these strings, we will define an auxiliary function $\sigma$.

For $\tau$ an $N_r^a$-strategy, $\sigma(\tau) = \seq{i}$ will be the beginning of the $f_i$ or $g_i$ that $\tau$ creates.  For $\tau$ an $N_{r,n}^a$-strategy and $\beta$ an outcome of $\tau$, $\sigma(\tau, \beta)$ will be the string which $\tau$ provides to strategies extending $\tau\cat\beta$.

Finally, for $\tau$ a $U_{i,e}$-strategy, if $\tau$ takes outcome $1$, we will define $\sigma(\tau, \theta)$ for each $\theta \subset \tau$ such that $\theta$ is a mother node beginning $f_i$, or $\theta$ is a mother node beginning $g_j$ with $j < i$.  These will be the strings which $N_{r,n}^a$-strategies extending $\tau\cat 1$ inherit, if they are daughters of such a $\theta$.

Also, for $\tau$ an $N_r^a$-strategy, we will define $v(\tau) = i$ for the $i$ such that $\tau$ is beginning $f_i$ or $g_i$.

Finally, if a $U_{i,e}$-strategy $\tau$ steals strings from its descendants, we will define $\ell(\tau)$ to be the length of the longest stolen string.  This will be used in building the priority tree, later.

\smallskip

{\em Choosing a pair.}  During the construction, pairs $(\sigma, a) \in \omega^{<\omega} \times\{0, 1\}$ will be chosen by strategies; we will say that $\tau$ {\em chooses} $(\sigma, a)$ when it declares that $(\sigma, a)$ is growing (see below).  This is similar to strategies choosing strings in the previous construction.  However, unlike the previous construction, a strategy may choose more than one pair.  Also, a pair may be chosen by a strategy and then later chosen by another strategy.  When this happens, the first strategy will never again be visited, and so we may think of it as the pair being transferred from the first strategy to the second.  As in the previous strategy, if a pair $(\sigma, a)$ with $\sigma \in s^{<s}$ has not been chosen by any strategy by the start of stage $s$, it will never be chosen by any strategy.  Unlike the previous construction, the pairs $(\seq{}, a)$ are never chosen, and so we will handle them explicitly in $G$.

\smallskip

{\em Growing a pair.} At a stage $s$, a strategy may declare that the pair $(\sigma, a)$ is growing.  When this happens, we let $m$ be largest such that we have already declared $S_m( (\emptyset, \sigma, a) )$ (or $m = -1$ if no such declarations have been made).  We declare $S_{k}( (\emptyset, \sigma, a))$ for $0 \le k < m+2$.  We also declare $S_{k}( (F, \sigma, a))$ for every $F \subseteq s$ and $0 \le k < m$.

\smallskip

{\em Strategy for $G$.} Our action is much the same as the action for the global strategy in \Cref{thm:cc_high_Scott}.  At the end of stage $s$, for every $\sigma \in s^{<s}$ and every $a < 2$, declare $S_0( (\emptyset, \sigma) )$.  Also, at the end of every stage, grow the pairs $(\seq{}, 0)$ and $(\seq{}, 1)$.

\smallskip

{\em Strategy for $N_r^a$.} Suppose $\tau$ is an $N_r^a$-strategy.  At the first stage when $\tau$ is visited, an $i \in \omega$ larger than any number previously mentioned in the construction is chosen.  $\tau$ makes the definitions $\sigma(\tau) = \seq{i}$, $v(\tau) = i$.

At every stage when $\tau$ is visited, grow the pair $(\seq{i}, a)$.

$\tau$ always takes the outcome $\outcome$.  

\smallskip

{\em Strategy for $N_{r,n}^a$.}  Suppose $\tau$ is an $N_{r,n}^a$-strategy visited at stage $s$.  Let $\theta$ be the mother node for $\tau$, i.e.\ the unique $\theta \subset \tau$ which is an $N_r^a$-strategy.  Let $\rho \subset \tau$ be largest such that one of the following holds:
\begin{itemize}
\item $\rho = \theta$.  In this case, let $\gamma = \sigma(\theta)$.
\item $\rho$ is an $N_{r,n-1}^a$-strategy.  In this case, let $\alpha$ be the outcome such that $\rho\cat\alpha \subseteq \tau$, and let $\gamma = \sigma(\rho, \alpha)$.
\item $a = 0$ and $\rho$ is a $U_{i,e}$-strategy with $i = v(\theta)$.  In this case, let $\gamma = \sigma(\rho, \theta)$.
\item $a = 1$ and $\rho$ is a $U_{i,e}$-strategy with $i > v(\theta)$.  In this case, let $\gamma = \sigma(\rho, \theta)$.
\end{itemize}
We will argue in the verification that $|\gamma| = n$.

Let $k$ be the number of times $\tau$ has previously taken outcome $\infty$, and let $t<s$ be the last stage at which $\tau$ took outcome $\infty$ (or $t = 0$ if there is no such stage).  If there is a $q \in [t, s)$ such that $\phi(n, q)$ holds, we will take outcome $\infty$ at stage $s$; otherwise, we will take outcome $k$.

Regardless, let $\beta$ be the outcome we are taking at stage $s$.  If $\sigma(\tau, \beta)$ is not already defined, we define $\sigma(\tau, \beta) = \gamma \cat s$.  We grow the pair $(\sigma(\tau,\beta), a)$ and then take outcome $\beta$.

\smallskip

{\em Strategy for $U_{i,e}$.}  Suppose $\tau$ is a $U_{i,e}$-strategy visited at stage $s$.  If there is no $\theta \subset \tau$ an $N_r^0$-strategy with $v(\theta) = i$, then $\tau$ never acts and simply takes outcome 0 at every stage.

Suppose instead there is such a $\theta$.  Let 
\[
C_\tau = \{\theta\} \cup \{ \psi \subset \tau : \text{$\psi$ is an $N_r^1$-strategy with $v(\psi) < i$}\}.
\]
By the recursion theorem, there is some $x_\tau \in \omega$ such that we can control the enumeration of $x_\tau$ into $\emptyset'$.  We initially keep $x_\tau$ out of $\emptyset'$.

At stage $s$, we check if there is a set $\{\pi_\psi, \alpha_\psi : \psi \in C_\tau\}$ such that:
\begin{itemize}
\item Each $\pi_\psi$ extends $\tau\cat 0$ and is a daughter of $\psi$;
\item Each $\alpha_\psi$ is an outcome of $\pi_\psi$;
\item $\sigma(\pi_\psi, \alpha_\psi)$ is defined for every $\psi$; and
\item $\Phi_{e,s}(\bigoplus_{\psi \in C_\tau} \sigma(\pi_\psi,\alpha_\psi); x_\tau)\converge = 0$.
\end{itemize}
If there is no such set, $\tau$ takes outcome $0$.

If there is, we fix such a set $\{\pi_\psi, \alpha_\psi : \psi \in C_\tau\}$.  We will henceforth always take outcome $1$, and we will always use this set.  We enumerate $x_\tau$ into $\emptyset'$ and define $\sigma(\tau, \psi) = \sigma(\pi_\psi,\alpha_\psi)$ for each $\psi \in C_\tau$.  We also define $\ell(\tau) = \max\{ |\sigma(\tau,\psi)| : \psi \in C_\tau \}$.

If $\tau$ is being visited at stage $s$ after having fixed a set, we grow $(\sigma(\tau,\theta), 0)$ and $(\sigma(\tau,\psi), 1)$ for $\psi \in C_\tau - \{\theta\}$, and then take outcome 1.

\smallskip

{\em Notation.}  For notational convenience, if $\tau$ is an $N_r^a$-strategy, we define $a(\tau) = a$.

\smallskip

{\em Strategy for $M_i$.} This is mostly as in the proof of \Cref{thm:cc_high_Scott}, with the obvious changes for pairs $(\sigma, a) \in \omega^{<\omega}\times\{0,1\}$ rather than strings $\sigma \in \omega^{<\omega}$.

$C_\tau$ is defined to consist of those pairs $(\sigma,a)$ such that there is a $\rho \subset \tau$ for which one of the following holds:
\begin{itemize}
\item $\rho$ is an $N_r^a$-strategy and $\sigma = \sigma(\tau)$.
\item $\rho$ is an $N_{r,n}^a$-strategy, $\rho\cat \alpha \subseteq \tau$ and $\sigma = \sigma(\rho,\alpha)$.
\item $\rho$ is a $U_{i,e}$-strategy with $\rho\cat 1 \subseteq \tau$, and there is $\psi \in C_\rho$ with $a = a(\psi)$ and $\sigma = \sigma(\rho,\psi)$.
\end{itemize}
It is important to note that this is finite and is completely determined by the first stage at which $\tau$ is visited.

If $t < s$ is the last stage at which $\tau$ took an infinite outcome and $k_0$ is the number of times $\tau$ has taken an infinite outcome, $B_\tau(s)$ consists of every pair $(\sigma, a) \in t^{<t} \times \{0,1\}$ except for those in $C_\tau$ and those chosen by a strategy extending $\tau\cat k_0$.

Otherwise, this is as in the proof of \Cref{thm:cc_high_Scott}. 

\smallskip

{\em Some auxiliary definitions.} For a mother strategy $\psi$ and a $\tau \supset \psi$, we let $n(\tau, \psi)$ be the largest $n$ such that $\psi$ has an $N_{r,n}^a$-daughter strategy along $\tau$, or $n(\tau,\psi) = 0$ if there is no such $n$.

Suppose $\psi$ is an $N_r^a$-strategy, and $\tau \supset \psi$ is a $U_{i,e}$-strategy with $\psi \in C_\tau$.  For every $n$ with $n(\tau,\psi) < n < |\sigma(\tau,\psi)|$, $\tau$ has defined $f_i(n)$ or $g_j(n)$ for the $f_i$ or $g_j$ started by $\psi$.  In this case, we say that $\tau$ {\em blocks} $N_{r,n}^a$.  Observe that a given $\tau$ can only block finitely many $N_{r,n}^a$.

\smallskip

{\em Running the construction and the priority tree.}  As mentioned, our priority tree will be constructed dynamically.  Our reason for doing this is that a $U_{i,e}$-strategy along the true path may define $f_i$ on some interval $I$.  To maintain the desired property of $f_i+g_j$, we must ensure that $g_j(x)$ is defined by $N$-strategies, for every $j > i$ with $j < x$ and $x \in I$.  So the appropriate $N$-strategies must occur on the true path before the next $U$-strategy.  Similar concerns about the definition of $f_i$ apply.  Also, if a $U_{i,e}$-strategy defines $f_i(n)$, then we cannot have an $N$-strategy which seeks to define $f_i(n)$.

At stage $s$, we begin by visiting the strategy $\seq{}$.  After having visited a strategy $\tau$ at stage $s$ with $|\tau| < s$, we let $\alpha$ be the outcome taken by $\tau$ at this stage.  We next visit the strategy $\tau\cat \alpha$. After having visited a strategy $\tau$ with $|\tau| =s$, we run the global strategy and end the stage.

The first time a given node on the priority tree is visited, we will determine what type of strategy it is.  This will determine what outcomes it has, and thus what children it has on the priority tree.  But we will not decide what type each of the children is until they are visited.

Suppose we are visiting node $\tau$ for the first time.  We make $\tau$ a $\zeta$-strategy, where $\zeta$ is least in the priority ordering subject to the following constraints:
\begin{itemize}
\item There can be no $\rho \subset \tau$ which is a $\zeta$-strategy;
\item $\tau$ cannot be an $N_{r,n}^a$-strategy if there is a $U_{i,e}$-strategy $\rho$ with $\rho\cat 1 \subseteq \tau$ and $\rho$ blocks $N_{r,n}^a$.
\item If there is a $U_{i,e}$-strategy $\rho$ with $\rho\cat 1\subseteq \tau$, and $\psi \subset \rho$ is a mother node, then $\tau$ cannot be a $U_{i',e'}$-strategy if this would make $n(\tau,\psi) \le \ell(\rho)$.
\end{itemize}

We put the lexicographic ordering on the tree of strategies, where we order the outcomes of $M_i$ as $\infty_\infty < \dots < \infty_2 < \infty_1 < \infty_0 < \dots < 2 < 1 < 0$, we order the outcomes of $N_{r,n}^a$ as $\infty < \dots < 2 < 1 < 0$, and we order the outcomes of $U_{i,e}$ as $1 < 0$.  Observe that our construction has the property that if a strategy $\tau$ is visited, and then at a later stage a strategy $\rho$ which is lexicographically to the left of $\tau$ is visited, $\tau$ can never again be visited.

\subsection*{The true path} We define the true path recursively, maintaining the inductive hypothesis that every strategy on the true path is visited infinitely often.  $\seq{}$ is on the true path.  If $\tau$ is on the true path, let $\alpha$ be the leftmost outcome which $\tau$ takes infinitely often.  By our actions for the various strategies, this necessarily exists.  Then $\tau\cat \alpha$ is on the true path.  We call $\alpha$ the {\em true outcome} of $\tau$.

We let $TP$ denote the true path.

By our earlier observation, no strategy off the true path is visited infinitely often.  Observe that the $TP$ is computable from $\zerojj$.

\subsection*{Verification}We begin by showing that our dynamic construction did not omit any strategies we did not intend to.  We say that $N_{r,n}^a$ is {\em blocked on the true path} if there is a $U_{i,e}$-strategy $\rho$ with $\rho\cat1 \in TP$ and $\rho$ blocks $N_{r,n}^a$.

\begin{claim}
For every strategy type $\zeta$ which is not blocked on the true path, there is a $\zeta$-strategy on the true path.
\end{claim}

\begin{proof}
Towards a contradiction, let $\zeta$ be the least type in the priority ordering which does not occur on the true path and is not blocked.  Fix $\rho$ on the true path such that all types before $\zeta$ occur along $\rho$.  If $\zeta$ is not of type $U_{i,e}$, then since $\zeta$ is not blocked, the next strategy along the true path must be of type $\zeta$.

If instead $\zeta$ is of type $U_{i,e}$, then since there is no $\zeta$-strategy on the true path, there can be no $U_{i',e'}$-strategy on the true path for any $U_{i',e'}$ later in the priority ordering than $\zeta$.  So no strategy along the true path extending $\rho$ will have type $U_{i',e'}$.  So only finitely many $N_{r,n}^a$ are blocked.  So for sufficiently long $\tau$ on the true path, the relevant $n(\tau, \psi)$ will have grown sufficiently large that $\zeta$ is no longer constrained from being chosen for the next strategy, contrary to our choice of $\zeta$.
\end{proof}

\begin{claim}
If $\tau$ is an $N_{r,n}^a$-strategy on the tree of strategies, then the $\gamma$ defined in the action for $\tau$ has length $n$.
\end{claim}

\begin{proof}
By induction on $n$ and our action for the $N_{r,n}^a$-strategies, as well as our priority ordering and the dynamic construction of the tree of strategies.
\end{proof}

\begin{claim}
If $\psi$ is a mother node on the true path, then $\psi$ has infinitely many daughter nodes on the true path.
\end{claim}

\begin{proof}
By our dynamic construction of the tree of strategies.  Note that we are not allowed to select a $U_{i,e}$-strategy until $n(\tau, \psi) > \ell(\rho)$ for every mother node $\psi \subset \rho$.
\end{proof}

Let $X$ be the set of $v(\psi)$, for $\psi$ an $N_r^0$-strategy along the true path.  Let $Y$ be the set of $v(\psi)$, for $\psi$ an $N_r^1$-strategy along the true path.  For $\psi$ an $N_r^0$-strategy along the true path with $v(\psi) = i$, define $f_i = \bigcup \sigma(\tau,\alpha)$, where the union ranges over $\tau\cat \alpha \in TP$ with $\tau$ a daughter of $\psi$.  Similarly, for $\psi$ an $N_r^1$-strategy along the true path with $v(\psi) = j$, define $g_j = \bigcup \sigma(\tau,\alpha)$, where the union ranges over $\tau\cat \alpha \in TP$ with $\tau$ a daughter of $\psi$.  By the previous claim, these are elements of $\omega^\omega$.  Further, as $TP$ is $\zerojj$-computable, so are $f_i$ and $g_j$, for $i \in X$, $j \in Y$.

Define $T_0 = \{ \sigma : \exists i \in X\, \sigma \subset f_i\}$ and $T_1 = \{ \sigma : \exists j \in Y\, \sigma \subset g_j\}$.  The argument that the relationship of the $S_n$ to $T_0$ and $T_1$ is as previously described is the same as the corresponding argument in the proof of \Cref{thm:cc_high_Scott}.

The argument that $\A$ is computably categorical is also as in the proof of \Cref{thm:cc_high_Scott}.  The reader might worry that the possibility for a string $\sigma$ to be chosen by multiple strategies could cause trouble.  However, note that the only way in which $\sigma$ can be chosen by multiple strategies is if it is chosen first by some strategy $\pi$, and then later chosen by a $U_{i,e}$-strategy $\rho$ with $\rho\cat 0 \subseteq \pi$.  If $\tau$ is an $\M_i$-strategy, then $\tau$'s action does not depend on precisely which strategy has chosen $\sigma$---rather, it only depends on which outcome of $\tau$'s is extended by the strategy that has chosen $\sigma$.  If $\tau$ is on the true path, and some strategy $\pi$ extending $\tau\cat\alpha$ has chosen $\sigma$, then every strategy which chooses $\sigma$ extends $\tau\cat\alpha$, and so $\tau$ is untroubled by this behavior.

It remains only to show that the $f_i$ and $g_j$ behave as we would like them to.

\begin{claim}
For $i < j < n$ with $i \in X$ and $j \in Y$, at most one of $f_i(n)$ and $g_j(n)$ is defined by some $U_{i,e}$-strategy $\rho$ with $\rho\cat 1$ along the true path.
\end{claim}

\begin{proof}
Suppose $f_i(n)$ is defined before $g_j(n)$ on the true path, and some $\rho$ a $U_{i,e}$-strategy with $\rho\cat 1$ along the true path defines $f_i(n)$.  Let $\psi$ be the $N_r^1$-strategy with $v(\psi) = j$ and $\psi \in TP$.  It cannot be that $\rho\cat 1 \subseteq \psi$, because then $\psi$ would have chosen $j > n$.  So $\psi\subset \rho$, and thus a daughter $N_{r,n}^1$-strategy is guaranteed to occur on the true path by our dynamic construction of the tree of strategies.

The argument for the case where $g_j(n)$ is first defined is the same.
\end{proof}

To show that $Z \leq_T f_i + g_j$ for each $i < j$, we need only show the following:

\begin{claim}
For all $i < j < n$ with $i \in X$ and $j \in Y$, if $n \not \in Z$, then $\neg \exists s > (f_i + g_j)(n)\, \phi(n, s)$.
\end{claim}

\begin{proof}
By the previous claim, without loss of generality $f_i(n)$ is defined by some $N_{r,n}^0$-strategy.  Let $\tau$ be the strategy along the true path defining $f_i(n)$.  Let $s_0$ be largest such that $\phi(n, s_0)$.  Then $\tau$ has outcome $k$ for some $k \in \omega$, and it first takes outcome $k$ at some stage $s_1 > s_0$.  So $f_i(n)$ is defined to be some $s \ge s_1$, by construction.
\end{proof}

\begin{claim}
For all $e \in \omega$ and $i \in X$, if $j_0 < j_1 < \dots j_{k-1}$ are the elements of $Y$ less than $i$, then $\emptyset' \neq \Phi_e(f_i \oplus g_{j_0} \oplus \dots \oplus g_{j_{k-1}})$.
\end{claim}

\begin{proof}
Fix $\tau$ the $U_{i,e}$-strategy along the true path.  Since $N_r^a$-strategies choose their number large, $i \in X$ implies there is a $\theta \in C_\tau$ with $a(\theta) = 0$, $v(\theta) = i$.  Similarly, for each $m < k$ there is a $\psi_m \in C_\tau$ with $a(\psi_m) = 1$, $v(\psi_m) = j_m$.

If $\tau$ has true outcome 0, then $x_\tau \not \in \emptyset'$, and
\[
\neg [\Phi_e(f_i \oplus g_{j_0} \oplus \dots \oplus g_{j_{k-1}}; x_\tau)\converge = 0],
\]
as otherwise there would be some collection of sufficiently large $\pi_\psi \supseteq \tau\cat 0$ and $\alpha_\psi$ for $\psi \in C_\tau$ with $\pi_\psi\cat\alpha_\psi$ along the true path such that $\Phi_e( \bigoplus_{\psi \in C_\tau} \sigma(\pi_\psi,\alpha_\psi); x_\tau)\converge = 0$, 
and $\tau$ would instead have true outcome 1.

If $\tau$ has true outcome 1, then $x_\tau \in \emptyset'$, $f_i \supset \sigma(\tau, \theta)$ and $g_{j_m} \supset \sigma(\tau, \psi_m)$ for $m < k$, and
\[
\Phi_e( \sigma(\tau, \theta)\oplus \sigma(\tau, \psi_0) \oplus \dots \oplus \sigma(\tau, \psi_{k-1}); x_\tau)\converge = 0
\]
by construction.
\end{proof}

This completes the proof of \Cref{thm:weak_deg_cat}.

\section{Closing thoughts}

In the construction for \Cref{thm:cc_high_Scott}, given a computable tree $T$, we constructed a computably categorical structure and a tree $Q$ which is $\Delta^0_3$-isomorphic to $T$ such that $[Q]$ was coded into the automorphism group of our structure in a particular fashion.  The reason for $T$ and $Q$ being $\Delta^0_3$-isomorphic rather than computably isomorphic is because constructing the isomorphism requires the true path of the construction, and the computable categoricity machinery we used requires a $\Delta^0_3$ true path.  If there were some way to avoid this obstacle and make $T$ and $Q$ computably isomorphic (or even equal), then these methods would be applicable to a wide range of questions.  For example, we would immediately be able to conclude that lowness for categoricity and lowness for isomorphism are the same.

Similarly, the reason $\zerojj$ was chosen as the degree of categoricity in \Cref{thm:weak_deg_cat} was because we again required a $\Delta^0_3$ true path for the computable categoricity machinery, and the true path was necessary to compute the various $f_i$ and $g_j$.  We suspect that, using standard techniques, this could be changed to any $\Sigma^0_2$ degree strictly above $\zeroj$.  We do not know, however, whether the same result could be repeated with $\zeroj$.

\newcommand{\etalchar}[1]{$^{#1}$}

\end{document}